\documentclass[11pt]{amsart}

\usepackage{mathtools}

\usepackage{microtype} 
\usepackage[hidelinks, pagebackref]{hyperref}
\usepackage[dvipsnames]{xcolor}
\usepackage{marginnote}
\usepackage{tikz}

\makeatletter
\define@key{href}{font}{#1}
\makeatother
\usepackage{xpatch}
\newcommand\hrefdefaultfont{\ttfamily}
\xpatchcmd\href{\setkeys{href}{#1}}{\setkeys{href}{font=\hrefdefaultfont,#1}}{}{\fail}

\renewcommand*{\backref}[1]{}
\renewcommand*{\backrefalt}[4]{
  \ifcase #1 
  [No citations.]
  \or [#2]
  \else [#2]
  \fi }

\usepackage{cite}
\makeatletter
\newcommand{\citecomment}[2][]{\citenum{#2}#1\citevar}
\newcommand{\citeone}[1]{\citecomment{#1}}
\newcommand{\citetwo}[2][]{\citecomment[,~#1]{#2}}
\newcommand{\citevar}{\@ifnextchar\bgroup{;~\citeone}{\@ifnextchar[{;~\citetwo}{]}}}
\newcommand{\citefirst}{\@ifnextchar\bgroup{\citeone}{\@ifnextchar[{\citetwo}{]}}}

\makeatother

\newcommand{\refthm}[1]{Theorem~\ref{Thm:#1}}
\newcommand{\reflem}[1]{Lemma~\ref{Lem:#1}}
\newcommand{\refprop}[1]{Proposition~\ref{Prop:#1}}
\newcommand{\refcor}[1]{Corollary~\ref{Cor:#1}}

\newcommand{\refdef}[1]{Definition~\ref{Def:#1}}
\newcommand{\refsec}[1]{Section~\ref{Sec:#1}}

\newcommand{\from}{\colon} 

\newcommand{\CC}{{\mathbb{C}}}
\newcommand{\HH}{{\mathbb{H}}}
\newcommand{\RR}{{\mathbb{R}}}

\newcommand{\QQ}{{\mathbb{Q}}}
\newcommand{\PP}{{\mathbb{P}}}

\newcommand{\ZZ}{{\mathbb{Z}}}

\newcommand{\calA}{{\mathcal{A}}}

\newcommand{\Dev}{{\mathcal{D}}}
\newcommand{\calL}{{\mathcal{L}}}
\newcommand{\calS}{{\mathcal{S}}}
\newcommand{\calT}{{\mathcal{T}}}
\newcommand{\calV}{{\mathcal{V}}}

\newcommand{\isom}{\cong}
\newcommand{\Isom}{\operatorname{Isom}}
\newcommand{\PSL}{\operatorname{PSL}}

\newcommand{\bdy}{\partial}
\newcommand{\image}{\mathrm{Im}}
\newcommand{\VolT}{{\textsc{Vol3}}}
\newcommand{\mdoubleoseven}{{\texttt{m007}}}
\newcommand{\mooneo}{{\texttt{m010}}}
\newcommand{\thsup}{\mathrm{th}}

\renewcommand{\ss}{\mathbf{s}}

\newcommand{\rr}{\mathbf{r}}
\newcommand{\zz}{\mathbf{z}}
\newcommand{\cover}{\widetilde}
\newcommand{\Vol}{\operatorname{Vol}}

\makeatletter
\let\c@equation\c@subsection
\makeatother
\numberwithin{equation}{section} 

\makeatletter
\let\c@figure\c@equation
\makeatother
\numberwithin{figure}{section} 

\makeatletter
\let\c@table\c@equation
\makeatother
\numberwithin{table}{section}

\numberwithin{equation}{section}
\theoremstyle{plain}
\newtheorem{theorem}[equation]{Theorem}
\newtheorem{lemma}[equation]{Lemma}
\newtheorem{proposition}[equation]{Proposition}
\newtheorem{corollary}[equation]{Corollary}

\newtheorem{conjecture}[equation]{Conjecture}

\theoremstyle{definition}
\newtheorem{definition/}[equation]{Definition}

\newenvironment{definition}
  {%
   \pushQED{\qed}\begin{definition/}}
  {\popQED\end{definition/}}

\long\def\@savemarbox#1#2{\global\setbox#1\vtop{\hsize\marginparwidth 
  \@parboxrestore\tiny\raggedright #2}}
\begin{document}

\title{Hyperbolic manifolds without positive spun triangulations}

\author[D.~Futer]{David Futer}
\address[]{Department of Mathematics, Temple University,
Philadelphia, PA 19122, USA}
\email[]{dfuter@temple.edu}

\author[J.~Purcell]{Jessica S.~Purcell}
\address[]{School of Mathematics, Monash University, Clayton, VIC 3800, Australia }
\email[]{jessica.purcell@monash.edu}

\author[S.~Schleimer]{Saul Schleimer}
\address[]{Department of Mathematics, 
University of Warwick, Coventry CV4 7AL, UK}
\email[]{s.schleimer@warwick.ac.uk}

\begin{abstract}
Using a result of Choi, we provide the first examples of pairs $(N, \gamma)$
where $N$ is a closed hyperbolic three-manifold,
where $\gamma$ is a simple closed geodesic in $N$, and
where $N - \gamma$ has no positive spun ideal triangulation.
In our first two examples, $N = \VolT$ and the geodesics are its systole and second systole. 
This provides evidence for the conjecture that \VolT\ has no positive spun ideal triangulation for any choice of $\gamma$. 
\end{abstract}

\thanks{This work is in the public domain.}





\maketitle

\section{Introduction}

Suppose that $N$ is a closed connected oriented hyperbolic three-manifold.
Suppose that $\gamma$ is a simple closed geodesic in $N$.
We say that the pair $(N, \gamma)$ admits a \emph{positive spun triangulation} 
if there is an ideal triangulation $\calT = (t_i)$ of $N - \gamma$ and hyperbolic shapes $(z_i)$ for the tetrahedra $t_i$ so that
\begin{itemize}
\item 
$\image(z_i) > 0$ for all $i$ and 
\item 
the metric completion of $\calT$, equipped with the shapes $(z_i)$, is isometric to $N$. 
\end{itemize}

Here is our main result:

\begin{theorem}
\label{Thm:Holonomies}
Let $N$ be $\VolT = \textup{\mdoubleoseven}(3, 1)$, the third manifold in the SnapPy census of closed oriented hyperbolic three-manifolds. 
Let $\gamma$ be a shortest or a second-shortest geodesic in $N$. 
Then $(N, \gamma)$ does not admit a positive spun triangulation.
\end{theorem}

Overall, starting from the one-cusped SnapPy census~\cite{SnapPy}, 
we found 271 examples of pairs $(N, \gamma)$, including the two above,
where $N$ is a closed hyperbolic three-manifold,
where $\gamma$ is a simple closed geodesic in $N$, and
where numerical evidence strongly suggests that $N$ has no positive spun triangulation along $\gamma$.
In similar fashion, we found 28 examples of pairs $(Q, \gamma)$
where $Q$ is a closed hyperbolic orbifold,
where $\gamma$ is the singular locus of $Q$, with cone angle $\pi$ in each case, and
where $Q$ has no positive spun triangulation along $\gamma$.
Finally, we found three examples of pairs $(N, \gamma)$ 
as above where, additionally, the polytope of angle structures is non-empty. 
This implies that a particular approach to the generalised Casson conjecture cannot succeed. 
See \refsec{Further} for details.

\subsection{History and context}

A longstanding conjecture in geometric topology posits that every noncompact, finite-volume hyperbolic three-manifold admits a decomposition into positive (that is, geometric) ideal tetrahedra.
Yoshida~\cite{Yoshida:TriangConjecture} gave, perhaps, the first instance of this conjecture in the literature and some evidence for it.
See also Wada, Yamashita, and Yoshida~\cite{wada1996inequality}, Petronio \cite[Conjecture 2.3]{Petronio:IdealTriangulations}, and Petronio and Porti~\cite{PetronioPorti} for early accounts of the conjecture and its history.

Positive ideal triangulations are known to exist for certain families of manifolds: 
\begin{itemize}
\item
the manifolds of SnapPy's cusped census (Hildebrand and Weeks~\cite{HildebrandWeeks}, extended by Callahan, Hildebrand and Weeks~\cite{CHW}, Burton~\cite{Burton} and Li~\cite{Li}),
\item
punctured torus bundles and two-bridge links (Gu\'eritaud and Futer~\cite{GueritaudFuter}),
\item
highly twisted manifolds (Ham and Purcell~\cite{HamPurcell}),
\item gluings of convex ideal polyhedra with at most six faces (Sirotkina~\cite{Sirotkina}) or ideal dodecahedra
 (Goerner~\cite{Goerner:Platonic}).
\end{itemize}
Dadd and Duan~\cite{DaddDuan} produce infinitely many positive ideal triangulations of the figure-eight knot complement.

As additional evidence for the conjecture, Luo, Schleimer, and Tillmann proved that every cusped hyperbolic three-manifold has a finite cover with a positive triangulation~\cite{LuoSchleimerTillmann}. 
Futer, Hamilton, and Hoffman strengthened this to produce a cover admitting infinitely many positive ideal triangulations~\cite{FHH}.

There is also evidence against the conjecture.
Recall that every cusped three-manifold admits a decomposition into convex ideal polyhedra~\cite{EpsteinPenner}.
However, it is not known whether this can be further subdivided into positive tetrahedra (see, for example, the discussion of~\cite{PetronioPorti}). 
Very recently Li~\cite{Li:PersonalCommunication} found cusped hyperbolic three-manifolds none of whose minimal ideal triangulations is positive. 
Note, however, that these examples do admit positive ideal triangulations with one more tetrahedron.

Following Thurston \cite{Thurston79}, we may regard cusped manifolds as limits of closed three-manifolds, orbifolds, and cone-manifolds.
In this sense, examples of the latter not having positive spun triangulations is again evidence against the conjecture. 
Choi~\cite{Choi:PosOriented} found the first such cone manifold.
Our main result (\refthm{Holonomies}) gives the first such examples of closed hyperbolic manifolds (equipped with a choice of geodesic).
This supports the following: 

\begin{conjecture}
\label{Conj:Vol3}
Suppose that $N$ is the hyperbolic manifold \VolT.
Then, for any embedded closed geodesic $\gamma$, the pair
 $(N, \gamma)$ does not admit a positive spun triangulation.
\end{conjecture}

This conjecture was first hinted at by Hodgson and Weeks in their construction of the SnapPy closed census in the early 1990s~\cite{Hodgson:PersonalCommunication}.
It has been asked by Trnkova~\cite{Trnkova:PersonalCommunication};
for a related conjecture see~\cite[Problem~5.15]{DelpHoffossManning}.
In addition to \refthm{Holonomies}, this conjecture is supported by extensive computer searches, none of which succeeded in finding a positive spun triangulation. Note that \VolT\ is the only hyperbolic manifold encountered during computational searches that has this property; see \refsec{OtherManifolds} for more details. 

The manifold \VolT\ has appeared in other settings.
It was discovered by Hodgson and Weeks, who proved it was hyperbolic by producing a Dirichlet domain (unpublished).
It was investigated in some detail by Jones and Reid~\cite{JonesReid}; 
they prove it is hyperbolic using arithmetic techniques.
It was conjectured to be the third smallest volume hyperbolic three-manifold by Gabai, Meyerhoff, and Thurston~\cite{GMT}, who describe it topologically as the $(3, 1)$ Dehn filling of \mdoubleoseven. 
This conjecture about volumes was recently proven by Gabai, Haraway, Meyerhoff, Thurston, and Yarmola~\cite{GHMTY}. 
It was also certified (again) to be hyperbolic, using Mostow rigidity and a three-fold cover, in HIKMOT~\cite{HIKMOT}.
It is the only closed hyperbolic three-manifold that has no closed geodesic contained in an embedded tube of radius $\log(3)/2$, by work begun in~\cite{GMT, JonesReid} and completed by Gabai and Trnkova~\cite{GabaiTrnkova}.

\subsection{Acknowledgements}

This work has benefited from many enlightening conversations, including with the following individuals: Young Choi, Nathan Dunfield, Matthias Goerner, Craig Hodgson, Neil Hoffman, Steve Kerckhoff, Henry Segerman, Stephan Tillmann, and Maria Trnkova. 
We also thank Diane Maclagan for explaining the various features of the software package Macaulay2 \cite{Macaulay2}.

This paper is based upon work supported by the National Science Foundation of the United States. 
Grant No.\ DMS--2405046 supported Futer. 
Grant No.\ DMS--2424139 supported Purcell and Schleimer while they were in residence at the Simons Laufer Mathematical Sciences Institute in Berkeley, California in early 2026.
Purcell was additionally supported by the Australian Research Council, grant DP240102350.

\section{Background and tools}

In this section we discuss some of the tools that are needed to prove \refthm{Holonomies}.
References for the background material include:
Thurston~\cite[Chapter~4]{Thurston79},
Neumann--Zagier~\cite[Section~2]{NeumannZagier85},
Futer--Gu\'eritaud~\cite[Section~2]{FuterGueritaud11}, and
Purcell~\cite[Chapter~4]{Purcell20}.

\subsection{Ideal triangulations}\label{Sec:IdealTri}

Suppose that $M$ is a compact, connected, oriented, three-manifold with non-empty boundary.
Suppose that all components of $\bdy M$ are tori.

Suppose that $\calT$ is a collection of oriented model tetrahedra, together with (orientation reversing) face pairings.
We say that $\calT$ is an \emph{ideal triangulation} of $M$ if the realisation space of $\calT$, minus its vertices, is homeomorphic (in orientation-preserving fashion) to the interior of $M$.
Every model tetrahedron has six model edges.
Collections of model edges are glued in the realisation space to give edges in the quotient. 

\subsection{Shape varieties}
\label{Sec:Shape}

Suppose that $\calT$ is an ideal triangulation of $M$, with $M$ as above.
For every model tetrahedron $t_i \in \calT$ we take $z_i$ to be a variable over $\CC - \{0, 1\}$.
We define the \emph{edge parameters} for $t_i$ as follows:
\[
z_i = z_i
\qquad
z'_i = \frac{1}{1 - z_i} 
\qquad
z''_i = \frac{z_i - 1}{z_i}
\]
Non-adjacent edges of $t_i$ share the same parameter.
Suppose that $v$ is any vertex of $t_i$ and that $T(v)$ is a small triangle in $t_i$ cutting off $v$.
We call $T(v)$ a \emph{cusp triangle}.
We label the corners of $T(v)$ with the edge parameter of the associated edge.
We arrange the labels so that there is a vertex $v$ of $t_i$ where, viewing $T(v)$ from $v$, we have (in anticlockwise order) the edge parameters $z_i$, $z'_i$, and $z''_i$.
See~\cite[Figure~1, page~311]{NeumannZagier85}.
Note that there are two conventions on the ordering of the edge parameters.
Here we follow the convention used by SnapPy;
see~\cite[Figure~15]{Weeks05}.

For every edge $e$ in $\calT$, we take the product of the edge parameters of the model edges identified to $e$.
Setting this product equal to $1$ gives the \emph{edge equation} for $e$.
Note that there are $|\calT|$ edges in the quotient;
thus there are $|\calT|$ edge equations. 
For more detailed expositions, see
Thurston~\cite[Chapter~4.2]{Thurston79},
Neumann--Zagier~\cite[page~312, Equation~16]{NeumannZagier85},
Futer--Gu\'eritaud~\cite[Definition~2.4]{FuterGueritaud11}, and
Purcell~\cite[Section~4.2]{Purcell20}.

We define the \emph{shape variety} $\calS(\calT)$ to be the set of solutions
\[
\zz = (z_i) \in (\CC - \{0, 1\})^{|\calT|}
\]
to the edge equations.
Note that this is a (quasi-projective) variety. 

Suppose that $\zz \in \calS(\calT)$ is a collection of shapes.
We place the vertices of some tetrahedron $t_0$ at $\infty$, $0$, $1$, and $z_0$.
This specifies a \emph{developing map} $\Dev \from \cover{M} \to \HH^3$
(see \cite[Chapter~3]{Purcell20}).
The developing map in turn determines the \emph{holonomy representation}:
\[
\rho_\zz \from \pi_1(M) \to \Isom^+(\HH^3) \isom \PSL(2, \CC)
\]
The conjugacy class of $[\rho_\zz]$ is a point of the \emph{character variety} $X(M)$.
The induced map from $\calS(\calT)$ to $X(M)$, taking $\zz$ to $[\rho_\zz]$, is algebraic;
see~\cite[Lemma~2.3]{Tillmann12} for further details.

A point $\zz \in \calS(\calT)$ is called \emph{positive} if every $z_i$ has positive imaginary part. Geometrically, this means that all tetrahedra are positively oriented, and that these positively oriented tetrahedra can be glued together to yield a (possibly incomplete) hyperbolic metric on $M$. 

\subsection{Slopes and spun triangulations}

Suppose that $M$ is a cusped three-manifold and $\calT$ an ideal triangulation, as above.
The cusp triangles for the vertices of the tetrahedra of $\calT$ fit together to triangulate the \emph{cusp tori} of $M$. 
(These are naturally copies of the components of $\bdy M$.)
Suppose that $T_j$ is a cusp torus.
An isotopy class of an essential simple curve in a $T_j$ is called a \emph{slope} in $T_j$.
An isotopy class of an oriented slope is called a \emph{signed slope}.  
Note that signed slopes are in bijection with primitive homology classes in $H_1(T_j)$.
For each cusp torus $T_j$, we pick a pair of signed slopes $\mu_j$ and $\lambda_j$, called the \emph{meridian} and \emph{longitude}, that form a basis for $H_1(T_j)$.

Suppose that $J$ is a subset of the set of cusp tori of $M$.
Suppose that $\ss = (s_j)_{j \in J}$ is a tuple of slopes, indexed by $J$, where $s_j$ is a slope in $T_j$.
We use $M(\ss)$ to denote the manifold obtained from $M$ by filling each $T_j$ (for $j \in J$) along the corresponding slope $s_j$.
If $J$ is empty, then $M(\ss) = M$;
if $J$ is the set of all slopes, then $M(\ss)$ is closed.

Suppose that the interior of $M$ admits a complete finite-volume hyperbolic structure.
Let 
\[
\rho^M \from \pi_1(M) \to \Isom^+(\HH^3)
\]
be the resulting discrete and faithful representation. 

Fix now a collection of slopes $\ss$.
Suppose that the filling $M(\ss)$ admits a hyperbolic structure.
Let $\psi_\ss \from \pi_1(M) \to \pi_1(M(\ss))$ be the surjective homomorphism induced by inclusion $M \hookrightarrow M(\ss)$.
The discrete and faithful representation associated to $\pi_1(M(\ss))$ is denoted $\rho^{M(\ss)} \from \pi_1(M(\ss)) \to \Isom^+(\HH^3)$.

\begin{definition}
\label{Def:Unsigned}
Suppose that $\ss$ is a collection of slopes for $M$. 
Suppose that $\calT$ is an ideal triangulation of $M$. 
Suppose that $\zz \in \calS(\calT)$ is a collection of shapes.
We say that $(\calT, \zz)$ is an \emph{unsigned spun triangulation} of $M(\ss)$ if
\[
[\rho^{M(\ss)} \circ \psi_\ss] = [\rho_\zz]
\]
That is, the two representations agree, up to a global conjugation.
\end{definition}


When the collection of slopes $\ss$ is signed, we require a stronger property. 

\begin{definition}
Suppose that $\ss$ is a collection of signed slopes for $M$. 
Let $\gamma_j$ be the core curve of the $j^\thsup$ solid torus in $M(\ss)$.
Note that the orientation on $s_j$ gives an orientation on $\gamma_j$, using the right hand-rule.

Suppose that $(\calT, \zz)$ is an unsigned spun triangulation of $M(\ss)$.
Recalling that $M(\ss)$ is hyperbolic, we identify its universal cover with $\HH^3$, and ensure that the images of the developed tetrahedra of $\cover{\calT}$ land in this copy of $\HH^3$.
Suppose that $v$ is an ideal vertex of a developed tetrahedron, say belonging to an preimage of the cusp $T_j$.
Thus $\bdy \Dev(v)$ lands at one of the two ends of some preimage of $\gamma_j$.
We say that $(\calT, \zz)$ is a \emph{(signed) spun triangulation} of $M(\ss)$ if
every image of every ideal vertex lands at the positive end of the corresponding geodesic.
\end{definition}

\subsection{Holonomies}
\label{Sec:Holonomies}
We now recall the \emph{holonomy} and the \emph{log holonomy} of a signed slope $s$ in a fixed cusp $T = T_j$.
For more details, see \cite[Definition~2.4]{FuterGueritaud11} or \cite[Section~4.3]{Purcell20}.

We continue to assume that $\zz \in \calS(\calT)$ is a collection of shapes that solves the edge equations.
We obtain a (two-dimensional) triangulation of $T$ by gluing together the cusp triangles in $\calT$ meeting $T$.
We isotope $s$ in $T$ to ensure that the arcs of $s$ enter and exit the cusp triangles through distinct edges, cutting off corners of triangles. 
Let $z_1, \dots, z_k$ be the shape parameters at these corners. 
We define the \emph{log holonomy} of $s$ to be
\[ 
H(s) = \sum_{i=1}^k \epsilon_i \log(z_i) 
\]
where $\epsilon_i = 1$ for corners of triangles to the left of $s$, and $\epsilon_i = -1$ for corners of triangles to the right.
We need not commit ourselves to a particular branch of the logarithm,  
because we will  always work with the derivative of the log holonomy.
(If we want to pin down the correct branch, we begin with the branch that, at the complete hyperbolic structure, makes all log holonomies of slopes vanish.)

The \emph{holonomy of $s$} is defined to be $h(s) =  \exp(H(s))$. 
Note that $h(s)$ is a rational function of the shapes $z_j$. 
Since $\zz$ lies in the shape variety, the holonomy $h(s)$ depends only on the isotopy class of $s$ and not on the choice of simplicial curve representing $s$.
In fact, $h \from H_1(T) \to \CC^*$ is a homomorphism.

If $(\calT, \zz)$ is an unsigned spun triangulation of $M(\ss)$, we have $\psi_\ss(s_j) = 1 \in \pi_1(M(\ss))$, which implies $\rho_\zz(s_j)$ is the identity, hence $h(s_j) = 1$. 
The \emph{holonomy equation} $h(s_j) = 1$ does not differentiate between $s_j$ and $-s_j$.
Holonomy equations will be used to cut out particular subvarieties of the shape variety $\calS(\calT)$.

If $\zz \in \calS(\calT)$ is a positive point (meaning every $z_i$ has positive imaginary part), and furthermore $(\calT, \zz)$ is a signed spun triangulation of $M(\ss)$, the positively oriented tetrahedra of $\calT$ fit together to give an incomplete hyperbolic metric on $M$ whose completion is $M(\ss)$. 
This recovers the definition of \emph{positive spun triangulation} in the opening paragraph of the paper. 
In this case, the log holonomy of $s_j$ is well-defined 
and satisfies $H(s_j)= 2\pi \sqrt{-1}$. 

\subsection{Positive triangulations and smoothness}

Our main tool in proving \refthm{Holonomies} is the following result of Choi~\cite[Corollary~4.14]{Choi:PosOriented}.

\begin{theorem}
\label{Thm:ChoiAlt}
Suppose that $M$ is a $k$--cusped hyperbolic three-manifold.
Suppose that $\calT$ is an ideal triangulation of $M$. 
Suppose that $\zz \in \calS(\calT)$ is a positive point (that is, the $z_j$ have positive imaginary parts).
Then the shape variety $\calS(\calT)$ is smooth at $\zz$.
Furthermore, for any tuple of slopes $\rr = (r_1, \ldots, r_k)$, with one slope $r_i$ per cusp of $M$,
the derivatives $d H(r_i)$ form a basis for the cotangent space of $\calS(\calT)$ at $\zz$. 
\end{theorem}

A key consequence of \refthm{ChoiAlt} is the following statement; 
compare \cite[Corollary~4.18]{Choi:PosOriented}.

\begin{corollary}
\label{Cor:ChoiAlt}
Suppose that $M$ is a $k$--cusped hyperbolic three-manifold.
Suppose that $\ss = (s_1, \ldots, s_k)$ is a tuple of slopes on the cusps of $M$, with one slope per cusp.
Suppose that $(\calT, \zz)$ is a spun ideal triangulation for $M(\ss)$, and that $r$ is a slope on some cusp of $M$ for which the derivative $d H(r)$ vanishes at $\zz \in \calS(\calT)$.
Then $M(\ss)$ does not admit a positive spun triangulation.
\end{corollary}

In Choi's work, \refcor{ChoiAlt} is stated in the special case where
$M$ is the complement of the figure-eight knot,
$\calT$ is its two-tetrahedron triangulation, and
$\ss$ is a particular cone-manifold filling. 
Her argument uses the smoothness of the character variety $X(M)$ as well as the fact that the algebraic map $\calS(\calT) \to X(M)$ is a local isomorphism at the given point. 
See \cite[Remark~(i) after Theorem~4.17]{Choi:PosOriented}.

To derive the full statement of \refcor{ChoiAlt} in the same manner, one would again need to check the smoothness of the character variety $X(M)$ as well as the quality of the algebraic map $\calS(\calT) \to X(M)$.
For completeness, we will give a full proof of \refcor{ChoiAlt} in the next subsection.
However, our proof uses the perspective of paths of essential triangulations~\cite{KalelkarSchleimerSegerman24} rather than character varieties.

\subsection{Essential triangulations}

\begin{definition}
Suppose that $M$ is a connected, compact, oriented three-manifold with non-empty boundary.
Suppose that $\cover{M}$ is its universal cover.
Let $\Delta_M$ be the set of boundary components of $\cover{M}$.

Suppose that $\calL$ is any set, equipped with an action of $\pi_1(M)$. 
The elements of $\calL$ are called \emph{labels}.
A function $L \from \Delta_M \to \calL$ is a \emph{labelling} if it is $\pi_1(M)$--equivariant.
\end{definition}

\begin{definition}
Suppose that $M$, $\calL$, and $L$ are as in the previous definition. 
Suppose that $\calT$ is an ideal triangulation of $M$. 
The triangulation $\calT$ is \emph{$L$--essential} if, for every edge $e$ of $\cover{\calT}$, the endpoints of $e$ have distinct labels (under $L$). 
\end{definition}

For a simple example of a labelling, take $\calL = \Delta_M$ and take $L$ to be the identity function.  
Here $L$--essentiality reduces to essentiality:
no edge is properly homotopic into  $\bdy \cover{M}$.

\begin{definition}
\label{Def:LabelFromShapes}
Suppose that $\calT$ is an ideal triangulation of $M$, and $\zz \in \calS(\calT)$.
We lift a tetrahedron $t_0$ to $\HH^3$ to have ideal points at $\infty$, $0$, $1$, and $z_0$. 
As described above, this induces a developing map $\Dev \from \cover{M} \to \HH^3$.
The endpoints of the lifted tetrahedra give a collection $\calL$ of points in $\bdy_\infty \HH^3$. 
The endpoint map $L \from \Delta_M \to \calL$ is the \emph{labelling induced by the shapes $\zz$}. 
\end{definition}

Recall that none of the shapes $z_i$ are allowed to be $0$, $1$, or $\infty$.
We deduce that a labelling $L$ induced by the shapes makes $\calT$ an $L$--essential triangulation.

The following is a consequence of a result of Kalelkar, Schleimer, and Segerman~\cite[Theorem~6.1]{KalelkarSchleimerSegerman24}.

\begin{theorem}
\label{Thm:KSS}
Suppose that $M$ is a $k$--cusped hyperbolic three-manifold. 
Suppose that $\ss$ is a tuple of signed slopes on the cusps of $M$.
Suppose that $M(\ss)$ is a hyperbolic three-manifold, and that $(\calT, \zz)$ and $(\calT', \zz')$ are spun triangulations for $M(\ss)$.
Then there is a birational map from the shape variety $\calS(\calT)$ to $\calS(\calT')$ that takes $\zz$ to $\zz'$.
\end{theorem}

\begin{proof}[Proof sketch]
As described in \refdef{LabelFromShapes}, the  developing map $\Dev \from \cover{M} \to \HH^3$ induces a map $\bdy \Dev \from \Delta_M \to \bdy \HH^3$. 
This is the desired labelling $L$, induced by the shapes $\zz$. 

Since $(\calT', \zz')$ is also a spun triangulation of $M(\ss)$, \refdef{Unsigned} implies that $\rho_\zz$ and $\rho_{\zz'}$ are conjugate representations.
In fact, since we have normalised the developing map into a standard position, the representations are equal. 
Furthermore, since $(\calT, \zz)$ and $(\calT', \zz')$ are both signed spun triangulations for the same oriented tuple $\ss$, the tips of the tetrahedra of $\calT'$ spin about the cores of the surgery solid tori in the same direction as those of $\calT$. 
Thus $(\calT, \zz)$ and $(\calT', \zz')$ yield the same labelling $L$.

Theorem~6.1 of~\cite{KalelkarSchleimerSegerman24} now gives a sequence of $L$--essential triangulations 
\[
\calT = \calT_0, \calT_1, \ldots, \calT_n = \calT'
\]
so that $\calT_{j+1}$ is obtained from $\calT_{j}$ by a 2--3, 3--2, 0--2, or 2--0 move.
Let $\calS_j = \calS(\calT_j)$ be the shape variety of $\calT_j$. 

Suppose by induction we have obtained shapes $\zz_j$ for $\calT_j$. 
Suppose further that $(\calT_j, \zz_j)$ is a spun triangulation of $M(\ss)$. 
Suppose finally that we have a birational map $\Phi_j \from \calS_0 \to \calS_j$ taking $\zz_0$ to $\zz_j$. 

We now carry out the induction step. 
Recall that $\calT_j$ is related to $\calT_{j + 1}$ via a move. 
Since both triangulations are $L$--essential, every edge of $\calT_{j+1}$ has distinct labels. 
So, under its developing map, every tetrahedron of $\calT_{j+1}$ has vertices at four distinct points in $\bdy \HH^3$.  
That is, the shapes $\zz_j$ induce shapes $\zz_{j+1}$. 
Furthermore, as described in Neumann and Yang \cite[proof of Proposition~10.1]{NeumannYang},
the shapes of the tetrahedra in $\calT_{j + 1}$ are rational functions of the 
shapes in $\calT_{j}$, and vice versa.
We may thus compose to obtain the next map $\Phi_{j+1}$, completing the induction. 
\end{proof}

Combining \refthm{ChoiAlt} with \refthm{KSS} yields a proof of \refcor{ChoiAlt}.

\begin{proof}[Proof of \refcor{ChoiAlt}]
Suppose that $r$ is the given slope so that $d H(r) = 0$ at $\zz \in \calS(\calT)$. 

Suppose now that $(\calT', \zz')$ is another spun triangulation of $M(\ss)$. 
By Theorem~\ref{Thm:KSS}, there is a birational map from $\calS(\calT)$ to $\calS(\calT')$ taking $\zz$ to $\zz'$. 
Thus $d H(r) = 0$ at $\zz' \in \calS(\calT')$. 
It follows that $d H(r)$ is not an element of a basis for the cotangent space at $\zz'$.
Thus, by Theorem~\ref{Thm:ChoiAlt}, the spun triangulation $(\calT', \zz')$ is not positive.
\end{proof}

\section{Spun triangulations of \VolT}\label{Sec:SpunVol3}

In this section, we consider spun triangulations of \VolT~ given by the Dehn fillings $\mdoubleoseven(3, 1)$ and $\mooneo(-1, 2)$.
We will find slopes on the boundary tori of \mdoubleoseven~ and \mooneo, respectively, such that in the hyperbolic structure given by the respective Dehn fillings, the derivative of the log holonomy is zero.
This and \refcor{ChoiAlt} implies that the corresponding Dehn filling does not admit a positive spun triangulation.

\begin{proposition}
\label{Prop:m007(3,1)}
Let $M$ be the SnapPy cusped census manifold \mdoubleoseven. 
The $(3, 1)$ Dehn filling of $M$ has no positive spun triangulation. Consequently, the pair $(\VolT, \gamma)$, where $\gamma$ is a shortest closed geodesic, has no positive spun triangulation.
\end{proposition}

The proof of \refprop{m007(3,1)} is subdivided into Lemmas~\ref{Lem:ShapeParamsM007}, \ref{Lem:DehnFillingEqnsM(3,1)}, and~\ref{Lem:ZeroDerivSlope_5/3}.
For all of these lemmas, we take $\calT$ to be the triangulation of $M = \mdoubleoseven$ provided by SnapPy. 
We begin by giving some defining equations for $\calS(\calT)$ and providing information about the shape parameters for the $(3, 1)$ Dehn filling.

\begin{lemma}
\label{Lem:ShapeParamsM007}
Suppose that $M$ is the SnapPy manifold \mdoubleoseven. 
Suppose that $\calT$ is the ideal triangulation of $M$ given in the SnapPy census. 
Suppose $x$, $y$, and $z$ are shape parameters for the three tetrahedra in $\calT$.
Then the shape variety $\calS(\calT)$ is cut out by:
\begin{align}
x y^2 z &= 1 \label{Eqn:Gluing1} \\
xz - x + y - z &= 0 \label{Eqn:Gluing2}
\end{align}
Furthermore, $\calS(\calT)$ is smooth.
\end{lemma}

\begin{proof}
Each of the three edges of $\calT$ gives an edge equation. 
Following \refsec{Shape}, these equations cut out the shape variety $\calS(\calT)$.
SnapPy encodes the three equations in a matrix:
\[
\begin{pmatrix}
1 & 0 & 0 & 2 & 0 & 0 & 1 & 0 & 0 \\
0 & 2 & 1 & 0 & 1 & 2 & 0 & 2 & 1 \\
1 & 0 & 1 & 0 & 1 & 0 & 1 & 0 & 1
\end{pmatrix}
\]
There is one row for each edge and one column for each shape parameter. 
An entry of $k$ in the $i$-th row and $j$-th column means that $j$-th shape parameter appears in the $i$-th edge equation with power $k$. 
For the convenience of the reader we recall that the shape parameters, in order, are as follows: 
\[
x, \quad \frac{1}{(1 - x)}, \quad \frac{(x - 1)}{x}, \qquad
y, \quad \frac{1}{(1 - y)}, \quad \frac{(y - 1)}{y}, \qquad
z, \quad \frac{1}{(1 - z)}, \quad \frac{(z - 1)}{z}
\]
Accordingly, the edge equations are as follows:
\begin{align*}
x \cdot y^2 \cdot z &= 1 \\
\frac{1}{(1 - x)^2} \cdot \frac{(x - 1)}{x} \cdot \frac{1}{(1 - y)} \cdot \frac{(y - 1)^2}{y^2} \cdot \frac{1}{(1 - z)^2} \cdot \frac{(z - 1)}{z} &= 1 \\
x \cdot \frac{(x - 1)}{x} \cdot \frac{1}{(1 - y)} \cdot z \cdot \frac{(z - 1)}{z} &= 1
\end{align*}
Note the sum of the entries of every column in the matrix is $2$.
Thus the product of the left-hand-sides of the edge equations is $1$. 
Thus we may omit one of the equations.
We omit the second equation and simplify the third to obtain:
\begin{align*}
x y^2 z &= 1 \\ 
xz - x + y - z &= 0 
\end{align*}
This gives equations~\eqref{Eqn:Gluing1} and~\eqref{Eqn:Gluing2}, which cut out the shape variety $\calS(\calT)$.

The derivatives of the left-hand-sides of equations~\eqref{Eqn:Gluing1} and~\eqref{Eqn:Gluing2} give the following cotangent vectors in $\CC^3$: 
\[
\begin{array}{rrrr}
(&\hspace{-10pt} y^2 z, & 2xyz, &  xy^2) \\  
(&\hspace{-10pt} z - 1, &    1, & x - 1)     
\end{array}
\]
Scaling does not alter their span; 
so we may divide by $y$ to obtain the following.
\[
\begin{array}{rrrr}  
(&\hspace{-10pt}    yz, & 2xz, &    xy) \\  
(&\hspace{-10pt} z - 1, &   1, & x - 1)     
\end{array}
\]
For $\calS(\calT)$ to be smooth, it suffices to show that these co-vectors are linearly independent at all points of $\calS(\calT)$. 
This condition will hold if the three maximal (two-by-two) minors have no common zero on $\calS(\calT)$. 
This can be checked by hand or using Macaulay2~\cite{Macaulay2}.
\end{proof}











\begin{lemma}
\label{Lem:DehnFillingEqnsM(3,1)}
Suppose the same hypotheses as in \reflem{ShapeParamsM007} and adopt the same notation. 
The solutions to~\eqref{Eqn:Gluing1} and~\eqref{Eqn:Gluing2} that satisfy the holonomy equation for the $(3, 1)$ Dehn filling of $M$, and have volume different from zero and $\pm \Vol(M)$, additionally satisfy:
\begin{align}
y^2 + y + 1 = 0 \label{Eqn:Relation3}
\end{align}
Furthermore, equations~\eqref{Eqn:Gluing1},~\eqref{Eqn:Gluing2} and~\eqref{Eqn:Relation3} have four solutions. 
\end{lemma}

\begin{proof}
Suppose that $s$ is the slope $3\mu + \lambda$ in $\bdy M$.
To obtain the shapes of the $(3, 1)$ filling, we add the holonomy equation $h(s) = 1$ to our list of conditions.
SnapPy again encodes the holonomies of the meridian and longitude in a matrix:
\[
\begin{pmatrix*}[r]
0 &            1 & \phantom{-}0 & -1 &  0 &  0 &  0 & \phantom{-}0 &  1 \\
2 & \phantom{-}0 &            1 &  0 & -1 & -1 & -1 &            0 & -1
\end{pmatrix*}
\]
The first row corresponds to the meridian $\mu$  and the second to the longitude $\lambda$.
The columns again correspond to the shape parameters.
Thus the holonomy equations are as follows:
\begin{align}
h(\mu) & = \frac{1}{(1 - x)} \cdot \frac{1}{y} \cdot \frac{(z - 1)}{z} \label{Eqn:HolonomyM} \\
h(\lambda) & = x^2 \cdot \frac{(x - 1)}{x} \cdot (1 - y) \cdot \frac{y}{(y - 1)} \cdot \frac{1}{z} \cdot \frac{z}{(z - 1)} 
     = x  \cdot (x - 1) \cdot (-y) \cdot \frac{1}{(z - 1)} \label{Eqn:HolonomyL}
\end{align}

Since $h \from H_1(T) \to \CC^*$ is a homomorphism, we compute $h(s)$ as follows:
\begin{align*}
h(s) 
  &= (h(\mu))^3 \cdot h(\lambda) \\
  &= \left( \frac{1}{(1 - x)} \cdot \frac{1}{y} \cdot \frac{(z - 1)}{z}\right)^3 \cdot \left( x \cdot (x - 1) \cdot (-y) \cdot \frac{1}{(z - 1)} \right) \\
  &= x \cdot \frac{1}{(1 - x)^2} \cdot \frac{1}{y^2} \cdot \frac{(z - 1)^2}{z^3}
\end{align*}
Recalling that we have set $h(s) = 1$ we compute as follows:
\begin{align}
x \cdot (z - 1)^2  & = (1 - x)^2 \cdot y^2 \cdot z^3  & \mbox{clear denominators} \nonumber \\
x^2 \cdot (z - 1)^2 & = (1 - x)^2 \cdot x \cdot y^2 \cdot z^3 & \mbox{multiply both sides by $x$} \nonumber \\
x^2 \cdot (z - 1)^2 & = (1 - x)^2 \cdot z^2 & \mbox{apply \eqref{Eqn:Gluing1} } \nonumber \\
\left( x \cdot (z - 1) \right)^2 & = \left( (x - 1) \cdot z \right)^2 \label{Eqn:HolonomyVol3}
\end{align}

We wish to take the square root of \eqref{Eqn:HolonomyVol3};
there are two cases depending on the choice of sign. 
Suppose first that $x (z - 1) = (x - 1) z$. 
Thus $x = z$. 
We substitute this into the equations \eqref{Eqn:Gluing1} and ~\eqref{Eqn:Gluing2} and obtain the following:
\begin{align}
x^2 y^2      &= 1  \label{Eqn:Gluing1Reduced} \\
x^2 - 2x + y &= 0  \label{Eqn:Gluing2Reduced}
\end{align}

We now take the square root of \eqref{Eqn:Gluing1Reduced} and break into subcases, again depending on the choice of sign. 
Suppose first that $xy = 1$. 
We substitute into \eqref{Eqn:Gluing2Reduced} to obtain 
\[
0 = x^3 - 2x^2 + 1 = (x - 1)(x^2 - x - 1)
\]
This polynomial has three real roots. 
So in this subcase all shape parameters are real. 
We deduce that the tetrahedra are flat (or degenerate). 
Thus the volume of the representation is zero (or undefined). 

Next, we assume $xy = -1$. 
In this subcase, substitution into \eqref{Eqn:Gluing2Reduced} gives $x^3 - 2x^2 - 1 = 0$. 
This equation has one real root and two complex-conjugate roots. 
The real root gives a representation with zero volume. 
The other pair of roots give the complete (unfilled) hyperbolic structure on $\mdoubleoseven$ and its complex conjugate.
(To see this, one checks that at these roots $h(\mu) = h(\lambda) = 1$, and that the volume interval contains the volume of $\mdoubleoseven$.
Recall that $h(\mu) = h(\lambda) = 1$ implies the structure is complete;
see~\cite[Proposition~4.15]{Purcell20}.)

The remaining possibility is that the correct square root of \eqref{Eqn:HolonomyVol3} is:
\begin{equation}\label{Eqn:HolonomySqrt}
(x - 1) \cdot z = -x \cdot (z - 1)
\end{equation}
Now equation \eqref{Eqn:Gluing2} says $xz - x - z = -y$ and \eqref{Eqn:HolonomySqrt} says $xz - x - z = -xz$. 
We deduce: 
\begin{equation}
  xz = y \qquad \mbox{and} \qquad x + z = 2y \label{Eqn:Relation1and2}
\end{equation}
Now equation \eqref{Eqn:Gluing1} implies $y^3 = 1$. 
We eliminate the solution with $y = 1$ to obtain \eqref{Eqn:Relation3}.

There are now two solutions for $y$; 
eliminating $z$ yields a quadratic in $x$. 
Thus there are four solutions.  
Additionally, one can check that these solutions have volume in a small interval containing $\pm v_3$, where $v_3 = 1.0149\dots$ equals the volume of \VolT.
\end{proof}

The four solutions of \reflem{DehnFillingEqnsM(3,1)} all correspond to complete hyperbolic structures on $\mdoubleoseven(3, 1)$.
These are spun triangulations with two orientations (of the manifold) and two directions of spinning about the Dehn filling geodesic. 
(This follows from \cite[Lemma 2.24]{KalelkarSchleimerSegerman24}.)

\begin{lemma}
\label{Lem:ZeroDerivSlope_5/3}
Suppose that $r$ is the slope $3\mu + 5\lambda$ in $\bdy M$.
Suppose that $p = (x, y, z)$ is one of the four points given by \reflem{DehnFillingEqnsM(3,1)}.
Then the derivative of the log holonomy $d H(r)$ vanishes
on the tangent space to $\calS(\calT)$ at $p$.
\end{lemma}

\begin{proof}
Recall by \reflem{ShapeParamsM007} that the shape variety is a smooth submanifold of $\CC^3$ cut out by edge equations \eqref{Eqn:Gluing1} and \eqref{Eqn:Gluing2}.
Taking the gradients of the left-hand sides of these equations, we find a pair of cotangent vectors that form a basis for the annihilator of the tangent space to $\calS(\calT)$.
We express these in vector notation as follows: 
\[
\begin{array}{rrrr r rrrr}  
yz\, dx &+& 2xz \, dy &+& xy \, dz & = (&
yz, & 2xz, &    xy) \\ 
(z-1)\, dx &+& dy &+& (x-1)\, dz & = (&\hspace{-10pt} z - 1, &   1, & x - 1)    
\end{array}
\]
At $p$ we have the equality $xz=y$ from \eqref{Eqn:Relation1and2}. We use this to simplify the basis as follows. Substitute $xz=y$ into the top, divide by $y$ (which is a nonzero constant at $p$), and subtract the bottom from the top, to obtain the new basis:
\begin{equation}\label{Eqn:SimplifyOrthoBasis}
\begin{array}{rrrr}  
(&\hspace{-10pt} z, & 2, & x) \\ 
(&\hspace{-10pt} 1, & 1, & 1)    
\end{array}
\end{equation}

Now we consider $d H(r)$, and show that it vanishes on the tangent space to $\calS(\calT)$ at $p$. 
Recall the holonomy equations \eqref{Eqn:HolonomyM} and \eqref{Eqn:HolonomyL}. 
Taking logs, we obtain:
\begin{align*}
H(\mu) &= - \log(1 - x) - \log(y)  + \log(z - 1) - \log(z) \\
H(\lambda) &= \log(x) + \log(x - 1) + \log(-y) - \log(z - 1)
\end{align*}
Differentiating gives:
\begin{align*}
d H(\mu) & = \left(-\frac{1}{x - 1}, \: -\frac{1}{y}, \: \frac{1}{z - 1} - \frac{1}{z} \right) \\
d H(\lambda) & = \left( \frac{1}{x} + \frac{1}{x - 1},  \: \frac{1}{y}, \: -\frac{1}{z - 1}               \right)
\end{align*}
Thus
    \begin{align*}
d H(r) 
  &= 3dH(\mu) + 5dH(\lambda) \\
      &= 
      \left( \frac{2}{x - 1} + \frac{5}{x}, \:
             \frac{2}{y}, \:
            -\frac{2}{z - 1} - \frac{3}{z} \right)
\end{align*}
In order to show that $dH(r)$ vanishes (on the tangent space to $\calS(\calT)$ at $p$), we show that it is a linear combination of the basis vectors \eqref{Eqn:SimplifyOrthoBasis}.
To do so, we take the determinant of the $3\times 3$ matrix given by the basis and $dH(r)$, and show the determinant vanishes. 
The determinant is:
\[
\left(-\frac{2z}{z - 1} - 3 - \frac{2z}{y} \right)
 + \left(\frac{4}{z - 1} + \frac{6}{z} + \frac{4}{x - 1} + \frac{10}{x} \right)
 + \left(\frac{2x}{y}  - \frac{2x}{x - 1} - 5 \right)
\]

Put this over a common denominator and divide by $2$. Rescale to consider the numerator alone, and use \eqref{Eqn:Relation1and2} to replace every instance of $xz$ by $y$. We obtain:
\[
-6y^3 + y^2(-16 + 11(x+z)) + 4y((x+z)-(x^2+z^2))
\]
Now use $x+z=2y$ from \eqref{Eqn:Relation1and2} and the formulas $x^2  = 2yx -y$ and $z^2 = 2yz-y$, which follow from \eqref{Eqn:Relation1and2}, to obtain:
\begin{align*}
-6y^3 + y^2 (-16 +22y) +4y( 2y -2yx + y -2yz + y) & = \\
-6y^3 +22y^3 -16y^2 +16y^2 -8xy^2 -8y^2z &= \\
16y^3 -8y^2(x+z) &= \\
16 y^3 -16 y^3 &= 0. \qedhere 
\end{align*}
\end{proof}

\begin{proof}[Proof of \refprop{m007(3,1)}]
By Lemma~\ref{Lem:ZeroDerivSlope_5/3}, at each point $p$ in the shape variety corresponding to the (3,1) Dehn filling of \mdoubleoseven, there exists a boundary slope $r$ for which the derivative of its log holonomy $dH(r)$ is zero. 
Thus, by \refcor{ChoiAlt}, there is no positive spun triangulation for $\mdoubleoseven(3,1)$.

Next, we recall from \cite{JonesReid} that the three-manifold $\mdoubleoseven(3,1)$ is \VolT.
This manifold and its shortest geodesics were studied in~\cite{GMT}. 
There are four shortest geodesics, permuted by the $\ZZ_2 \oplus \ZZ_2$ symmetry group.

We claim that the core $\gamma$ of the Dehn filling solid torus is one of these four shortest geodesics. 
Since computing rigorously in SnapPy requires a positive spun triangulation, and we just proved that no such triangulation exists, we cannot use SnapPy to verify this directly. 
Instead, we verify this using the cover of $\VolT$ that corresponds to a filling of the unique three-cusped triple cover of \mdoubleoseven. 
(Compare~\cite[page~74, proof of Theorem~1.2]{HIKMOT}.) 
This cover admits a positively oriented spun triangulation: for instance, the one with isosig 
$\texttt{gLLAQbcedefftoaolon\_DbCb(1,0)}$.
We use this triangulation to rigorously compute enough of the length spectrum to find all geodesics up to three times the length of $\gamma$, verifying that $\gamma$ is indeed shortest in $\VolT$. 
\end{proof}

\begin{proposition}
\label{Prop:m010(-1,2)}
Let $M$ be the SnapPy cusped census manifold \mooneo. 
The $(-1, 2)$ Dehn filling of $M$ has no positive spun triangulation. 
Consequently, the pair $(\VolT, \sigma)$, where $\sigma$ is a second-shortest closed geodesic, has no positive spun triangulation.
\end{proposition}

\begin{proof}
The claim that $N = \mooneo(-1, 2)$ is homeomorphic to $\VolT = \mdoubleoseven(3,1)$ is checked via SnapPy, by finding identical filled triangulations of the two closed manifolds.
Now, we check that the Dehn filling core $\sigma$ is one of the second-shortest geodesics.
As in the proof of \refprop{m007(3,1)}, we build the cyclic three-fold cover $\widehat N$ studied in HIKMOT~\cite{HIKMOT}.
Next, we observe that the rigorously verified second shortest geodesic in $\widehat N$ has length agreeing with $\sigma$, and rigorously check that no other geodesic in $\widehat N$ besides the systole could project to a shorter geodesic in $N$. 

Now, let $\calT$ be the three-tetrahedron SnapPy triangulation of \mooneo.
The rest of the proof follows a similar method to that of \refprop{m007(3,1)}. 
We will first obtain gluing equations in the shape parameters $x$, $y$, and $z$ that cut out the shape variety $\calS(\calT)$. 
We will then derive a third equation for the $s = (-1, 2)$ filling. 
Finally, we will compute $dH(r)$ for the slope $r = (2, -1)$ and show that $dH(r)$ vanishes (on the tangent space to $\calS(\calT)$ at every point corresponding to the Dehn filling along $s$). 
Since this is similar to \refprop{m007(3,1)}, we suppress some of the details.

The SnapPy triangulation $\calT$ of $M$ has three tetrahedra. 
As before, we simplify the SnapPy gluing equations in shape parameters $x$, $y$, and $z$ to obtain two equations that cut out the shape variety $\calS(\calT)$: 
\begin{align}
x^2 y^2 z^2 &= 1 \label{Eqn:GluingM010_1}\\
(x - 1) y^2 (z - 1) - (y - 1)^2 &=0 \label{Eqn:GluingM010_2}
\end{align}
Again, we  check using Macaulay2~\cite{Macaulay2} that these equations cut out a smooth variety. 

%
%
%
%
%

We now pin down the points of interest in the shape variety, corresponding to the $(-1, 2)$ Dehn filling. Again we find the holonomies for $\mu$ and $\lambda$ from SnapPy. 
 \begin{align*}
 h(\mu) &= \frac{1}{x^3} \cdot (x - 1)^2 \cdot \frac{1}{(y-1)^3} \cdot y \cdot (z - 1) \cdot \frac{1}{z} 
  \\
 h(\lambda) &= \frac{1}{x^2} \cdot \frac{1}{(y-1)^2} \cdot (z - 1)^2 \cdot \frac{1}{z^2} 
 \end{align*}

Recall that $s = -\mu + 2\lambda$ is the filling slope.  We 
substitute the holonomies $h(\mu)$ and $h(\lambda)$
into the equation  
$h(s)=h(\mu)^{-1}h(\lambda)^2=1$,
and use \eqref{Eqn:GluingM010_1} to obtain the simplified holonomy equation:
\begin{equation}\label{Eqn:SlopeEquation_M010}
x y (z - 1)^3 = (x - 1)^2  (y - 1)  z
\end{equation}

As in the proof of \reflem{DehnFillingEqnsM(3,1)}, we have two cases to analyse depending on which square root we take in \eqref{Eqn:GluingM010_1}.
First, we show the square root $xyz=1$ cannot give a point in the shape variety corresponding to the Dehn filling $(-1, 2)$.
Using this square root, and eliminating variables one at a time using \eqref{Eqn:GluingM010_2} and \eqref{Eqn:SlopeEquation_M010}, leads to only two solutions: $(x,y,z) = (1,1,1)$ and  $(x,y,z) = (-2, 1/4, -2)$. The first one is degenerate and the second corresponds to flat tetrahedra, so neither one can lead to a geometric structure.
Thus we must replace \eqref{Eqn:GluingM010_1} with the other square root:
\begin{equation}\label{Eqn:Gluing_M010_1_Root}
x y z = -1
\end{equation}
Using \eqref{Eqn:Gluing_M010_1_Root}, the other edge equation \eqref{Eqn:GluingM010_2} simplifies to a quadratic equation in $x$ and $z$:
\begin{equation}\label{Eqn:Gluing_M010_2_Root}
x^2z^2 + xz + x + z = 0 
\end{equation}

We now have the system of equations \eqref{Eqn:Gluing_M010_1_Root}, \eqref{Eqn:Gluing_M010_2_Root}, and \eqref{Eqn:SlopeEquation_M010}.
Eliminating variables one at a time, we find that $x$ satisfies
\[ (x + 1)(x^2 - x + 2)(x^4 - 6x^3 + 12x^2 + 6x + 3) = 0 \]
The term $(x+1)$ gives a solution with flat tetrahedra and zero volume, which we ignore.
The zeros of $x^2 - x + 2 = 0$ give the complete unfilled hyperbolic structure on $\mooneo$ with its two orientations, which we also ignore. 
Thus the desired solutions are zeros of the quartic.
Using this relation for $x$, we additionally find a simple relation for $y$.
We record them both.
\begin{equation}\label{Eqn:Relations_Simple}
x^4 - 6x^3 + 12x^2 + 6x + 3 = 0  \quad \mbox{ and } \quad y^2 = -\frac{1}{3}
\end{equation}
Each solution for $x$ in \eqref{Eqn:Relations_Simple} determines $y$ and $z$. 
Thus we get four solutions, each of which corresponds to a complete hyperbolic structure on $\mooneo(-1, 2)$.
These are spun triangulations with two orientations (of the manifold) and two directions of spinning about the Dehn filling geodesic. 

Next, let $p = (x,y,z)$ be one of the four solutions. We will show that for the slope $r = 2\mu - \lambda$, the derivative of the log holonomy $dH(r)$ vanishes at $p$. Observe that equation  \eqref{Eqn:Gluing_M010_1_Root} holds for all points of the component of $\calS(\calT)$ through $p$, because the value of $xyz$ cannot jump from $-1$ to $1$. 
Thus  \eqref{Eqn:Gluing_M010_1_Root} and \eqref{Eqn:Gluing_M010_2_Root} define this component. 
Differentiate the left-hand sides of the edge equations \eqref{Eqn:Gluing_M010_1_Root} and \eqref{Eqn:Gluing_M010_2_Root} 
to obtain a pair of cotangent vectors that form a basis for the annihilator of the tangent space to $\calS(\calT)$:
\begin{align*}
& (y z,\,  x z,\,  x y) \\
& (2xz^2 + z + 1, 0, 2x^2z + x + 1)
\end{align*}

We differentiate the log holonomy $H(r)$:
\[ dH(r) = \left( \frac{4}{x-1}-\frac{4}{x}, \frac{2}{y}+\frac{4}{1-y}, 0\right)
\]
As before, we check that $dH(r)$ lies in the annihilator by showing that the determinant of the matrix consisting of these three vectors is zero. 
The equations \eqref{Eqn:Gluing_M010_1_Root}, \eqref{Eqn:Gluing_M010_2_Root}, and \eqref{Eqn:Relations_Simple} imply that $x^2z^2 = -3$, hence
$z = \frac{3 - x}{x + 1}$ and therefore $y= \frac{x+1} {x(x-3)}$.
Substituting these into the determinant expression produces a multiple of the defining quartic \eqref{Eqn:Relations_Simple} for $x$, which equals zero.
Thus, by Corollary~\ref{Cor:ChoiAlt}, there can be no positive spun triangulation.
\end{proof}

Combining \refprop{m007(3,1)} with \refprop{m010(-1,2)} gives a proof of \refthm{Holonomies}. \qed

\section{Further computational results}\label{Sec:Further}

The two spun triangulations of \VolT\ that were described in the last section were found using a computer search with SnapPy. Here, we describe the method as well as the additional results that it produced.

\subsection{Our method}\label{Sec:Method}
Suppose that $M$ is a one-cusped hyperbolic manifold and $r$ is a slope on the cusp of $M$.
Observe that if $r = a\mu +b\lambda$, the equation  $dH(r) = 0$ can be written as
$a \, d H(\mu) + b \, d H(\lambda) = 0$. Since the tangent space to $\calS(\calT)$ has (complex) dimension one,
the equation $a \, d H(\mu) + b \, d H(\lambda) = 0$ is equivalent to 
\[
\langle a \, d H(\mu)  + b \, d H(\lambda), \, k \rangle = 0
\]
 for a single nonzero tangent vector $k$. This, in turn, is equivalent to
\begin{equation}\label{Eqn:Ratio}
\frac{ \langle dH(\mu), \, k \rangle}{\langle dH(\lambda), \, k \rangle} \: = \:  -\frac{b}{a} \: \in \: \QQ\PP^1
\end{equation}

We ran a search of the following type.
For a SnapPy census manifold $M$ with one cusp, equipped with the geometric framing $(\mu, \lambda)$, we considered the set of all not-too-long slopes $s$.
For the spun ideal triangulation of $M(s)$, we computed the left-hand side of \eqref{Eqn:Ratio}, and recorded the instances where this ratio was real (or infinite).
Curiously, in all cases where the ratio was real, it turned out to be rational.

\subsection{Other spun triangulations of \VolT}

According to our search, there are 87 Dehn fillings of census manifolds that give $\VolT$, including the fillings $\mdoubleoseven(3,1)$ and $\mooneo(-1, 2)$ that were studied in \refsec{SpunVol3}.
The manifolds that we found are contained in the ancillary files~\cite{FPS:Ancillary}. 
However, the technique of \refsec{SpunVol3} to rule out positive spun triangulations does not work for the 85 additional fillings.
For these Dehn fillings, SnapPy does not find any positive spun triangulations.
However, computing the ratio of \eqref{Eqn:Ratio} produced a non-real complex number, implying that
there is no slope $r$ for which $dH(r)$
vanishes at the corresponding point in the shape variety.
Thus Choi's theorem cannot be applied in these cases.

\subsection{Other pairs $(N, \gamma)$ without positive spun triangulations}\label{Sec:OtherManifolds}
We found a total of 271 pairs $(N, \gamma)$, where $N$ is a closed hyperbolic three-manifold, $\gamma$ is a closed geodesic, $M = N - \gamma$ is a census manifold, and numerical computations suggest that $(N, \gamma)$ does not admit any positive spun triangulation~\cite{FPS:Ancillary}. 
In each instance, the ratio in \eqref{Eqn:Ratio} was approximately a rational number  $-b/a$ with small denominator. 
Accordingly, we have reason to believe that $dH(a\mu+b\lambda) = 0$ at the corresponding point of the shape variety. 

However, our search did not produce any new examples of closed hyperbolic three-manifolds that might conjecturally have no positive spun triangulations.
For example, $\texttt{m389}(-1,1)$ has no positive spun triangulation, spun around the core of its Dehn filling. However, SnapPy identifies this manifold as isometric to $\texttt{m011}(1,3)$, which has a positive spun triangulation consisting of three tetrahedra, spun about a different geodesic.

\subsection{Orbifold pairs}
In addition to the 271 manifold pairs described above, we found 28 pairs $(Q,\gamma)$ where $Q$ is a closed hyperbolic three-orbifold, $\gamma$ is a singular geodesic of cone angle $\pi$, and the pair $(Q,\gamma)$ does not admit any positive spun triangulation. 
The 271 manifold pairs and 28 orbifold pairs are described in the ancillary files~\cite{FPS:Ancillary}.
Observe that since $\gamma$ is singular, there is no positive triangulation of $Q$ that is spun along a single geodesic.

\subsection{How not to prove Casson's conjecture}

Our examples also demonstrate the futility of a certain approach to the generalised Casson conjecture.
For more detailed background we refer to~\cite{FuterGueritaud11}.

Suppose $\calT$ is an ideal triangulation of a cusped finite-volume hyperbolic three-manifold $M$.
Suppose that $\ss$ is a (possibly empty) tuple of signed slopes, with at most one slope per boundary component of $M$. 
An \emph{angle structure} on $(\mathcal T, \ss)$ is a function $\alpha$ from the model edges of $\calT$ (as in \S\ref{Sec:IdealTri}) to the interval $(0,\pi)$ satisfying the following. 
\begin{enumerate}
\item
For each model tetrahedron $t$, we require that the three model edges meeting at any model vertex have angles summing to $\pi$. 
\item
For each edge $e$, the model edges gluing to give $e$ have angles summing to $2\pi$.
\item
For each signed slope $s \in \ss$ we proceed as in \refsec{Holonomies}:
the sum of the angles to the left of $s$, minus the sum of angles to its right, is $2\pi$.
\end{enumerate}
These linear (tetrahedron, edge, slope) equations are exactly the imaginary parts of the logarithmic (tetrahedron, edge, slope) equations.
However, angle structures need not give representations of $\pi_1(M)$.
Thus, having an angle structure gives less information than having a point of $\calS(\calT)$.
We use $\calA(\calT, \ss)$ to denote the resulting (open) polytope of angle structures.

Applying the Lobachevsky function to each angle and summing gives a volume functional $\calV \from \calA(\calT, \ss) \to \RR$, which is strictly concave down. 
Casson and  Rivin showed~\cite[Theorems~1.2 and 6.1]{FuterGueritaud11} that if $\calV$ has a local maximum $\alpha$ in  $\calA(\calT, \ss)$, then $\alpha$ is the set of angles of a positive spun triangulation of $M(\ss)$ (spun along the cores of the Dehn filling solid tori).
We deduce that for all other $\beta \in \calA(\calT, \ss)$ we have $\calV(\beta) < \Vol(M(\ss))$; compare \cite[Theorem 6.3]{FuterGueritaud11}.
This leads to the generalised Casson conjecture~\cite[Section 6.4]{FuterGueritaud11}.

\begin{conjecture}
Suppose that $M$ is a cusped finite-volume hyperbolic three-manifold. 
Suppose that $\ss$ is a collection of signed slopes.
Suppose that $\calT$ is any ideal triangulation of $M$.
Suppose that $\alpha$ is any point of $\calA(\calT, \ss)$. 
Then $\calV(\alpha) \leq \Vol(M(\ss))$. 
\end{conjecture}

One approach to the conjecture might start by finding a point of maximum volume in the compact polytope $\overline{\calA(\calT, \ss)}$. If the maximum occurs on the boundary, we might hope that this boundary point would guide us to a useful retriangulation of $\calT$.
One could then iterate the process, finding a sequence of triangulations and angle polytopes that ends at a positive triangulation.
But this is false hope.

For, consider the SnapPy manifold $\texttt{o9\_29517}$ with its given triangulation $\calT$.
Take $s = \mu + \lambda$. 
We compute to find the following: 
\begin{itemize}
\item
The open polytope $\calA(\calT, s)$ is non-empty.
\item
SnapPy gives shapes $\zz$ for $(\calT, s)$ so that the induced representation $\rho_\zz$ is discrete and faithful.
Of the nine tetrahedra of $(\calT, \zz)$, six are positive, two are negative, and one is flat. 
Thus the Casson--Rivin theorem~\cite[Theorem~6.1]{FuterGueritaud11} implies that $\calV$ has no maximum in $\calA(\calT, s)$. 
\item
The derivative of the log holonomy of the slope $r = 5\mu - 9\lambda$ vanishes at $\zz$. 
\end{itemize}
Hence by~\refcor{ChoiAlt} there is no way to retriangulate $\texttt{o9\_29517}(1,1)$ to obtain a positive spun triangulation. 
The same properties hold for $\texttt{o9\_21590}(-1,1)$ and $\texttt{o11\_465181}(1,1)$.

\renewcommand{\UrlFont}{\tiny\ttfamily}
\renewcommand\hrefdefaultfont{\tiny\ttfamily}
\bibliographystyle{plainurl}
\bibliography{bibfile}

\begin{thebibliography}{10}

\bibitem{Burton}
Benjamin~A. Burton.
\newblock The cusped hyperbolic census is complete.
\newblock {\em Trans. Amer. Math. Soc.}, to appear.
\newblock \href {https://doi.org/10.1090/tran/6767}
  {\path{doi:10.1090/tran/6767}}.

\bibitem{CHW}
Patrick~J. Callahan, Martin~V. Hildebrand, and Jeffrey~R. Weeks.
\newblock A census of cusped hyperbolic 3-manifolds.
\newblock {\em Math. Comp.}, 68(225):321--332, 1999.
\newblock \href {https://doi.org/10.1090/S0025-5718-99-01036-4}
  {\path{doi:10.1090/S0025-5718-99-01036-4}}.

\bibitem{Choi:PosOriented}
Young-Eun Choi.
\newblock Positively oriented ideal triangulations on hyperbolic
  three-manifolds.
\newblock {\em Topology}, 43(6):1345--1371, 2004.
\newblock \href {https://doi.org/10.1016/j.top.2004.02.002}
  {\path{doi:10.1016/j.top.2004.02.002}}.

\bibitem{SnapPy}
Marc Culler, Nathan~M. Dunfield, Matthias Goerner, and Jeffrey~R. Weeks.
\newblock Snap{P}y, a computer program for studying the geometry and topology
  of 3-manifolds.
\newblock Available at \url{http://snappy.computop.org} (January 2025).

\bibitem{DaddDuan}
Blake Dadd and Aochen Duan.
\newblock Constructing infinitely many geometric triangulations of the figure
  eight knot complement.
\newblock {\em Proc. Amer. Math. Soc.}, 144(10):4545--4555, 2016.
\newblock \href {https://doi.org/10.1090/proc/13076}
  {\path{doi:10.1090/proc/13076}}.

\bibitem{DelpHoffossManning}
Kelly Delp, Diane Hoffoss, and Jason~Fox Manning.
\newblock Problems in groups, geometry, and three-manifolds.
\newblock arXiv:1512.04620, 2015.
\newblock URL: \url{https://arxiv.org/abs/1512.04620}.

\bibitem{EpsteinPenner}
David B.~A. Epstein and Robert~C. Penner.
\newblock Euclidean decompositions of noncompact hyperbolic manifolds.
\newblock {\em J. Differential Geom.}, 27(1):67--80, 1988.
\newblock URL: \url{http://projecteuclid.org/euclid.jdg/1214441650}.

\bibitem{FuterGueritaud11}
David Futer and Fran\c{c}ois Gu\'eritaud.
\newblock From angled triangulations to hyperbolic structures.
\newblock In {\em Interactions between hyperbolic geometry, quantum topology
  and number theory}, volume 541 of {\em Contemp. Math.}, pages 159--182. Amer.
  Math. Soc., Providence, RI, 2011.

\bibitem{FHH}
David Futer, Emily Hamilton, and Neil~R. Hoffman.
\newblock Infinitely many virtual geometric triangulations.
\newblock {\em J. Topol.}, 15(4):2352--2388, 2022.
\newblock \href {https://doi.org/10.1112/topo.12271}
  {\path{doi:10.1112/topo.12271}}.

\bibitem{FPS:Ancillary}
David Futer, Jessica~S. Purcell, and Saul Schleimer.
\newblock Ancillary files stored with the {arXiv} version of this paper.
\newblock Available at \url{http://arxiv.org/}.

\bibitem{GHMTY}
David Gabai, Robert Haraway, Robert Meyerhoff, Nathaniel Thurston, and Andrew
  Yarmola.
\newblock Hyperbolic 3-manifolds of low cusp volume.
\newblock {\em Duke Math. J.}, to appear.
\newblock arXiv:2109.14570.

\bibitem{GMT}
David Gabai, G.~Robert Meyerhoff, and Nathaniel Thurston.
\newblock Homotopy hyperbolic 3-manifolds are hyperbolic.
\newblock {\em Ann. of Math. (2)}, 157(2):335--431, 2003.
\newblock \href {https://doi.org/10.4007/annals.2003.157.335}
  {\path{doi:10.4007/annals.2003.157.335}}.

\bibitem{GabaiTrnkova}
David Gabai and Maria Trnkova.
\newblock Exceptional hyperbolic 3-manifolds.
\newblock {\em Comment. Math. Helv.}, 90(3):703--730, 2015.
\newblock \href {https://doi.org/10.4171/CMH/368} {\path{doi:10.4171/CMH/368}}.

\bibitem{Goerner:Platonic}
Matthias Goerner.
\newblock Geodesic triangulations exist for cusped {P}latonic manifolds.
\newblock {\em New York J. Math.}, 23:1363--1367, 2017.
\newblock URL: \url{http://nyjm.albany.edu:8000/j/2017/23_1363.html}.

\bibitem{Macaulay2}
Daniel~R. Grayson and Michael~E. Stillman.
\newblock Macaulay2, a software system for research in algebraic geometry.
\newblock Available at \url{http://www2.macaulay2.com}.

\bibitem{GueritaudFuter}
Fran\c{c}ois Gu\'eritaud.
\newblock On canonical triangulations of once-punctured torus bundles and
  two-bridge link complements.
\newblock {\em Geom. Topol.}, 10:1239--1284, 2006.
\newblock With an appendix by David Futer.
\newblock \href {https://doi.org/10.2140/gt.2006.10.1239}
  {\path{doi:10.2140/gt.2006.10.1239}}.

\bibitem{HamPurcell}
Sophie~L. Ham and Jessica~S. Purcell.
\newblock Geometric triangulations and highly twisted links.
\newblock {\em Algebr. Geom. Topol.}, 23(3):1399--1462, 2023.
\newblock \href {https://doi.org/10.2140/agt.2023.23.1399}
  {\path{doi:10.2140/agt.2023.23.1399}}.

\bibitem{HildebrandWeeks}
Martin Hildebrand and Jeffrey Weeks.
\newblock A computer generated census of cusped hyperbolic 3-manifolds.
\newblock In {\em Computers and mathematics ({C}ambridge, {MA}, 1989)}, pages
  53--59. Springer, New York, 1989.

\bibitem{Hodgson:PersonalCommunication}
Craig Hodgson.
\newblock Personal communication, 2025.

\bibitem{HIKMOT}
Neil Hoffman, Kazuhiro Ichihara, Masahide Kashiwagi, Hidetoshi Masai, Shin'ichi
  Oishi, and Akitoshi Takayasu.
\newblock Verified computations for hyperbolic 3-manifolds.
\newblock {\em Exp. Math.}, 25(1):66--78, 2016.
\newblock \href {https://doi.org/10.1080/10586458.2015.1029599}
  {\path{doi:10.1080/10586458.2015.1029599}}.

\bibitem{JonesReid}
Kerry~N. Jones and Alan~W. Reid.
\newblock Vol3 and other exceptional hyperbolic 3-manifolds.
\newblock {\em Proc. Amer. Math. Soc.}, 129(7):2175--2185, 2001.
\newblock \href {https://doi.org/10.1090/S0002-9939-00-05775-0}
  {\path{doi:10.1090/S0002-9939-00-05775-0}}.

\bibitem{KalelkarSchleimerSegerman24}
Tejas Kalelkar, Saul Schleimer, and Henry Segerman.
\newblock Connecting essential triangulations {I}: via 2--3 and 0--2 moves.
\newblock arXiv:2405.03539, 2024.
\newblock URL: \url{https://arxiv.org/abs/2405.03539}.

\bibitem{Li}
Shana~Yunsheng Li.
\newblock The complete 10-tetrahedra census of orientable cusped hyperbolic
  3-manifolds.
\newblock arXiv:2512.02142, 2026.
\newblock URL: \url{https://arxiv.org/abs/2512.02142}.

\bibitem{Li:PersonalCommunication}
Shana~Yunsheng Li.
\newblock Personal communication, 2026.

\bibitem{LuoSchleimerTillmann}
Feng Luo, Saul Schleimer, and Stephan Tillmann.
\newblock Geodesic ideal triangulations exist virtually.
\newblock {\em Proc. Amer. Math. Soc.}, 136(7):2625--2630, 2008.
\newblock \href {https://doi.org/10.1090/S0002-9939-08-09387-8}
  {\path{doi:10.1090/S0002-9939-08-09387-8}}.

\bibitem{NeumannYang}
Walter~D. Neumann and Jun Yang.
\newblock Bloch invariants of hyperbolic 3-manifolds.
\newblock {\em Duke Math. J.}, 96(1):29--59, 1999.
\newblock \href {https://doi.org/10.1215/S0012-7094-99-09602-3}
  {\path{doi:10.1215/S0012-7094-99-09602-3}}.

\bibitem{NeumannZagier85}
Walter~D. Neumann and Don Zagier.
\newblock Volumes of hyperbolic three-manifolds.
\newblock {\em Topology}, 24(3):307--332, 1985.
\newblock \href {https://doi.org/10.1016/0040-9383(85)90004-7}
  {\path{doi:10.1016/0040-9383(85)90004-7}}.

\bibitem{Petronio:IdealTriangulations}
Carlo Petronio.
\newblock Ideal triangulations of hyperbolic 3-manifolds.
\newblock {\em Boll. Unione Mat. Ital. Sez. B Artic. Ric. Mat. (8)},
  3(3):657--672, 2000.

\bibitem{PetronioPorti}
Carlo Petronio and Joan Porti.
\newblock Negatively oriented ideal triangulations and a proof of {T}hurston's
  hyperbolic {D}ehn filling theorem.
\newblock {\em Expo. Math.}, 18(1):1--35, 2000.

\bibitem{Purcell20}
Jessica~S. Purcell.
\newblock {\em Hyperbolic knot theory}, volume 209 of {\em Graduate Studies in
  Mathematics}.
\newblock American Mathematical Society, Providence, RI, [2020] \copyright
  2020.
\newblock \href {https://doi.org/10.1090/gsm/209} {\path{doi:10.1090/gsm/209}}.

\bibitem{Sirotkina}
M.~L. Sirotkina.
\newblock On the triangulation of three-dimensional hyperbolic manifolds.
\newblock {\em Algebra i Analiz}, 14(6):192--204, 2002.

\bibitem{Thurston79}
William~P. Thurston.
\newblock {\em The Geometry and Topology of Three-Manifolds}.
\newblock Princeton Univ. Math. Dept. Notes, 1979.
\newblock Available at http://www.msri.org/communications/books/gt3m.

\bibitem{Tillmann12}
Stephan Tillmann.
\newblock Degenerations of ideal hyperbolic triangulations.
\newblock {\em Math. Z.}, 272(3-4):793--823, 2012.
\newblock \href {https://doi.org/10.1007/s00209-011-0958-8}
  {\path{doi:10.1007/s00209-011-0958-8}}.

\bibitem{Trnkova:PersonalCommunication}
Maria Trnkova.
\newblock Personal communication, 2026.

\bibitem{wada1996inequality}
Masaaki Wada, Yasushi Yamashita, and Han Yoshida.
\newblock An inequality for polyhedra and ideal triangulations of cusped
  hyperbolic 3-manifolds.
\newblock {\em Proc. Amer. Math. Soc.}, 124(12):3905--3911, 1996.
\newblock URL: \url{https://doi-org/10.1090/S0002-9939-96-03563-0}, \href
  {https://doi.org/10.1090/S0002-9939-96-03563-0}
  {\path{doi:10.1090/S0002-9939-96-03563-0}}.

\bibitem{Weeks05}
Jeff Weeks.
\newblock Computation of hyperbolic structures in knot theory.
\newblock In {\em Handbook of knot theory}, pages 461--480. Elsevier B. V.,
  Amsterdam, 2005.
\newblock \href {https://doi.org/10.1016/B978-044451452-3/50011-3}
  {\path{doi:10.1016/B978-044451452-3/50011-3}}.

\bibitem{Yoshida:TriangConjecture}
Han Yoshida.
\newblock Ideal tetrahedral decompositions of hyperbolic 3-manifolds.
\newblock {\em Osaka J. Math.}, 33(1):37--46, 1996.
\newblock URL: \url{http://projecteuclid.org/euclid.ojm/1200786689}.

\end{thebibliography}

\end{document}